\documentclass[a4paper,11pt,reqno]{article}  
\usepackage{amsmath,amsthm,amssymb,amscd,mathdots,stmaryrd,enumerate,shuffle,ytableau,varwidth,faktor,xfrac,xcolor,calc,titlepic,setspace,enumitem}
\usepackage[utf8]{inputenc}
\usepackage[nottoc]{tocbibind}
\usepackage{pst-node}
\usepackage{auto-pst-pdf}
\usepackage{tikz-cd} 
\usepackage{tikz}
\usetikzlibrary{arrows,positioning}
\usetikzlibrary{shadows}
\usetikzlibrary{calc}
\usetikzlibrary{shapes,snakes}
\usetikzlibrary{decorations}
\usetikzlibrary{decorations.pathreplacing}
\usepackage{hyperref}

\allowdisplaybreaks

\newlength{\bibitemsep}
\newlength{\bibparskip}
\let\oldthebibliography\thebibliography
\renewcommand\thebibliography[1]{%
  \oldthebibliography{#1}%
  \setlength{\parskip}{\bibitemsep}%
  \setlength{\itemsep}{\bibparskip}%
}

\newtheorem{theorem}{Theorem}
\newtheorem{proposition}[theorem]{Proposition}

\newtheorem{corollary}[theorem]{Corollary}

\numberwithin{equation}{section}
\theoremstyle{definition}
\newtheorem{definition}[theorem]{Definition}
\newtheorem{example}[theorem]{Example}
\newtheorem{remark}[theorem]{Remark}

\newcommand{\R}{\mathbb{R}}
\newcommand{\N}{\mathbb{N}}

\newcommand{\wt}{\operatorname{wt}}

\newcommand{\bone}{{\mathbf{1}}}

\newcommand{\bz}{\zeta_q^\mathrm{BZ}} 

\newcommand{\sz}{\zeta_q^\mathrm{SZ}} 

\newenvironment{itemize*}%
  {\begin{itemize}[topsep=4pt]%
    \setlength{\itemsep}{0pt}%
    \setlength{\parskip}{0pt}}%
  {\end{itemize}}
  
\newenvironment{enumerate*}%
  {\begin{enumerate}[topsep=4pt,wide,labelwidth=!,labelindent=0pt]%
    \setlength{\itemsep}{3pt}%
    \setlength{\parskip}{0pt}}%
  {\end{enumerate}}

\definecolor{darkblue}{RGB}{0,0,130}
\definecolor{mygray}{RGB}{160,160,160}
\definecolor{mycyan}{RGB}{150,255,255}
\definecolor{darkred}{RGB}{130,0,0}

\tikzset{
    >=stealth',
    punkt/.style={
           rectangle,
           rounded corners,
           draw=black, very thick,
           text width=6.5em,
           minimum height=2em,
           text centered},
    pil/.style={
           ->,
           thick,
           shorten <=2pt,
           shorten >=2pt,}
}

\makeatletter
\newcommand{\gettikzxy}[3]{%
  \tikz@scan@one@point\pgfutil@firstofone#1\relax
  \edef#2{\the\pgf@x}%
  \edef#3{\the\pgf@y}%
}
\makeatother

\title{Proving dualities for $q$MZVs with connected sums}
\author{Benjamin Brindle \thanks{The author has received funding from the European Research Council (ERC) under the European Union’s Horizon 2020 research and innovation programme (grant agreement No. 101001179).}}

\date{\today}
\AtEndDocument{\bigskip{\footnotesize%
  \textsc{University of Cologne, Department of Mathematics and Computer Sciences} \par  
  Weyertal 86-90, 50931 Cologne, Germany
  \par
  \textit{E-mail address:} \texttt{bbrindle@uni-koeln.de} \par
  \addvspace{\medskipamount}
}}
\begin{document}
\maketitle
\textsc{Abstract.} This paper gives a new application of so-called connected sums, introduced recently by Seki and Yamamoto \cite{SY}. Special about our approach is that it proves a duality for the Schlesinger– Zudilin and the Bradley–Zhao model of qMZVs simultaneously. The latter implies the duality for MZVs and the former can be used to prove the shuffle product formula for MZVs. Furthermore, the $q$-Ohno relations, a generalization of Bradley--Zhao duality, is also obtained.
\vspace{0.2cm}
\vspace{0.2cm}
\section{Notation and definitions}
For an \emph{admissible index} $\mathbf{k}=(k_1,\dots,k_r)\in\N^r$, i.e., $r\geq 0$ and $k_1\geq 2$, its \emph{multiple zeta value} (MZV) is
\begin{align*}
    \zeta(\mathbf{k}) := \sum\limits_{m_1>\dots>m_r>0} \frac{1}{m_1^{k_1}}\cdots \frac{1}{m_r^{k_r}}.
\end{align*}
To understand the algebraic structure of MZVs better on the one hand and to get connections to holomorphic functions, in particular, modular forms (see \cite{GKZ}, \cite{Bac}), on the other hand, it is useful to introduce $q$-analogs of MZVs. There are several models of $q$-analogs. We focus in this paper on the one by Schlesinger, and Zudilin and Bradley and Zhao. For further details on these and other models, we refer to \cite{Zha}, \cite{Bri}. In this note $q$ will be a formal variable or a real number with $0<q<1$.

The Bradley--Zhao model is defined as follows: Set $\bz(\emptyset):= 1$ and for $\mathbf{k} = (k_1,\dots,k_r)$ an admissible index define
\begin{align*}
    \bz(\mathbf{k}) := \sum\limits_{m_1>\dots>m_r>0} \frac{q^{m_1 (k_1 - 1)}}{(1-q^{m_1})^{k_1}}\cdots \frac{q^{m_r (k_r - 1)}}{(1-q^{m_r})^{k_r}}.
\end{align*}

Similar, but different, we define Schlesinger--Zudilin $q$MZVs via $\sz(\emptyset) := 1$ and for every \emph{SZ-admissible index} $\mathbf{k}$, i.e., $\mathbf{k}\in\N_0^r$ for some $r\geq 0$ with $k_1\geq 1$, we set
\begin{align*}
    \sz(\mathbf{k}) := \sum\limits_{m_1>\dots>m_r>0} \frac{q^{m_1 k_1}}{(1-q^{m_1})^{k_1}}\cdots \frac{q^{m_r k_r}}{(1-q^{m_r})^{k_r}}.
\end{align*}

\section{Dualities}
Write an admissible index $\mathbf{k}$ in the shape $\mathbf{k} = \left(k_1+1,\{1\}^{d_1-1},\dots,k_r+1,\{1\}^{d_r-1}\right)$ with $k_j,\, d_j\geq 1$ unique ($\{1\}^d$ means that $1$ is repeated $d$-times). For the next two theorems, we need the \emph{dual index},
\begin{align*}
    \mathbf{k}^\vee := \left(d_r+1,\{1\}^{k_r-1},\dots,d_1+1,\{1\}^{k_1-1}\right).
\end{align*}

\begin{theorem}[MZV-Duality, {\cite[§9]{Zag}}]
\label{MZV-dual}
For every admissible index $\mathbf{k}$, we have $\zeta\left(\mathbf{k}\right) = \zeta\left(\mathbf{k}^\vee\right)$.
\end{theorem}
The next theorem can be seen as a $q$-analogue of MZV-duality since MZV-duality follows immediately from it (cf. the proof):
\begin{theorem}[BZ-Duality, {\cite[Thm. 5]{Bra}}]
\label{BZ-dual}
For every admissible index $\mathbf{k}$, we have $\bz\left(\mathbf{k}\right) = \bz\left(\mathbf{k}^\vee\right)$.
\end{theorem}
A generalization of BZ-duality are the so-called $q$-Ohno relations, of which BZ-duality is the special case $c=0$:
\begin{theorem}[$q$-Ohno relation, {\cite[Thm. 1]{OT}}]
\label{q-Ohno}
For any admissible index ${\mathbf{k}} = (k_1,\dots,k_r)$ and any $c\in\N_0$ we have
\begin{align*}
    \sum\limits_{ |{\mathbf{c}}| = c} \bz\left({\mathbf{k}} + {\mathbf{c}}\right) = \sum\limits_{ |{\mathbf{c}}| = c} \bz\left({\mathbf{k}}^\vee + {\mathbf{c}}\right),
\end{align*}
where we sum over all ${\mathbf{c}} = (c_1,\dots,c_r)\in\N_0^r$ with $|{\mathbf{c}}| := c_1+\dots + c_r = c$.
\end{theorem}

For the SZ-model, we write an SZ-admissible index in the shape $\mathbf{k} = \left(k_1+1,\{0\}^{d_1},\dots,k_r+1,\{0\}^{d_r}\right)$ with $k_j,\, d_j\geq 0$ unique and define the \emph{SZ-dual index},
\begin{align*}
    \mathbf{k}^\dagger := \left(d_r+1,\{0\}^{k_r},\dots,d_1+1,\{0\}^{k_1}\right).
\end{align*}

\begin{theorem}[SZ-Duality, {\cite[Thm. 8.3]{Zha}}]
\label{SZ-dual}
For $\mathbf{k}$ SZ-admissible, we have $\sz\left(\mathbf{k}\right) = \sz\left(\mathbf{k}^\dagger\right)$.
\end{theorem}

Note that BZ- and SZ-duality on algebraic level look the same, both can be obtained by the same anti-automorphism on the non-commutative free algebra in two variables (see, e.g., \cite[Thm. 3.5, Thm. 3.16]{Bri}). But they imply different things. BZ-duality gives direct duality for MZVs, while SZ-duality does not. However, SZ-duality implies another important result in the theory of MZVs, namely the shuffle product formula (cf. \cite{EMS}, \cite{Sin}, for details \cite[Thm. 3.46]{Bri}).

For some calculations in the next section we need the connection between admissible and SZ-admissible index: An index $\mathbf{k}$ is admissible if and only if $\mathbf{k}-\bone$ is SZ-admissible ($\mathbf{k}+\bone$ is the index, which is $\mathbf{k}$ with every entry increased by $1$; similar for $\mathbf{k}-\bone$). Furthermore, we have for $\mathbf{k}$ admissible
\begin{align}
    (\mathbf{k}-\bone)^\dagger = \mathbf{k}^\vee - \bone.
\end{align}

\section{Connected sums \& proof of dualities}
\label{ssec:Connected Sums}

As a new tool for proving identities among ($q$-)multiple zeta values, Seki and Yamamoto introduced the concept of so-called connected sums (this notion is independent of connected sums in topology). With connected sums, they have proven, e.g., the duality of MZVs, Hoffman's identity, and the $q$-analogue of Ohno's relation, cf. \cite{Sek} or \cite{SY}.

Using connected sums, we give a new proof of the duality of Schlesinger--Zudilin $q$MZVs, the duality of Bradley--Zhao $q$MZVs and the usual duality of MZVs. It turns out that the connected sum defined below has the power to prove all three statements at once. As a by-product, we also get a proof for the $q$-Ohno relation. The proof is inspired by the one of Seki and Yamamoto (\cite{SY}), where the authors proved $q$-Ohno's relation for non-modified Bradley--Zhao $q$MZVs. We work with modified $q$MZVs, which will be here the reason that we can prove also Schlesinger--Zudilin duality at the same time. 

\begin{definition}[Connected sum]
Let be $r,\, s\geq 0$, $\mathbf{k} = (k_1,\dots,k_r)\in\N_0^r,\ \boldsymbol\ell = (\ell_1,\dots,\ell_s)\in\N_0^s$ and $x$ real with $|x|<1$. Define the connected sum as
\begin{align*}
    Z_q\left(\mathbf{k};\boldsymbol\ell;x\right) := \sum\limits_{\substack{m_1>\dots>m_r>m_{r+1}=0\\ n_1>\dots >n_s>n_{s+1}=0}}
    &\prod\limits_{i=1}^r \frac{q^{m_i k_i}}{(1 - q^{m_i} x)(1-q^{m_i})^{k_i}} 
    \prod\limits_{j=1}^s \frac{q^{n_j \ell_j}}{(1 - q^{n_j} x)(1-q^{n_j})^{\ell_j}}
    \\
    &\times\frac{q^{m_1 n_1}f_q(m_1;x) f_q(n_1;x)}{f_q(m_1 + n_1;x)},
\end{align*}
where $f_q(m;x) := \prod\limits_{h=1}^m (1 - q^h x)$.
\end{definition}

\begin{remark}\
\label{cs-rem}
\begin{enumerate*}
    \item The connected sum $Z_q$ is symmetric in ${\mathbf{k}}$ and ${\boldsymbol\ell}$ by definition.
    \item Notice that the connected sum is well-defined in the sense that it is a series over positive real numbers and hence either a positive real number (if convergent) or $+\infty$ (if not convergent).
    \item  If $k_1\geq 1$, then $Z_q\left({\mathbf{k}};\emptyset;0\right) = \sz\left(\mathbf{k}\right)$.
    \item If $k_1\geq 1$, then $\lim\limits_{x\rightarrow 1} Z_q\left({\mathbf{k}};\emptyset;x\right) = \bz\left({\mathbf{k + 1}}\right)$.
\end{enumerate*}
\end{remark}

\begin{proposition}[Boundary conditions]
If $k_1\geq 1,\, 0<q<1$ and $x\in\R$ with $|x|<1$, then $Z_q\left({\mathbf{k}};\emptyset ; x\right)$ is a well-defined real number.
\end{proposition}

\begin{proof}  One has
\begin{align*}
    Z_q\left({\mathbf{k}};\emptyset ; x\right) 
    = &\,
    \sum\limits_{m_1 > \dots > m_{r+1} := 0} \prod\limits_{i=1}^r \frac{q^{m_i k_i}}{(1-q^{m_i}x)(1-q^{m_i})^{k_i}} 
    \\ \leq &\,  
    \frac{1}{(1-q x)^r} \sum\limits_{m_1>\dots > m_{r+1} := 0} \prod\limits_{i=1}^r \frac{q^{m_i k_i}}{(1-q^{m_i})^{k_i}} 
    = \, 
    \frac{1}{(1-q x)^r}\,  \sz(k_1,\dots,k_r),
\end{align*}
which is well-defined since $k_1\geq 1$, i.e., $(k_1,\dots,k_r)$ SZ-admissible. \end{proof} 
After we have checked well-definednes of $Z_q$, we state and prove now distinguished relations among our connected sums.
\begin{theorem}[Transport relations]
\label{thm: transport rel}
Let be $r,s\geq 0$ and $k_1,\dots,k_r,\, \ell_1,\dots,\ell_s\geq 0$. If $s>0$,
\begin{align}
\label{T1}
    Z_q\left((0,k_1,\dots,k_r);(\ell_1,\dots,\ell_s);x\right)
    =
    Z_q\left((k_1,\dots,k_r);(\ell_1+1,\ell_2,\dots,\ell_s);x\right)
\end{align}
and if $r>0$,
\begin{align}
\label{T2}
    Z_q\left((k_1+1,k_2,\dots,k_r);(\ell_1,\dots,\ell_s);x\right)
    =
    Z_q\left((k_1,\dots,k_r);(0,\ell_1,\ell_2,\dots,\ell_s);x\right).
\end{align}
\end{theorem}

\begin{proof}  The second equality follows from the first by symmetry and the first one is obtained from
\begin{align*}
    &\sum\limits_{a>m} \frac{1}{1-q^a x} \frac{q^{a n} f_q(a;x) f_q(n;x)}{f_q(a+n;x)}
    \\
    =\, &\, \frac{q^n}{1-q^n}\sum\limits_{a>m} \left(\frac{q^{(a-1) n} f_q(a-1;x) f_q(n;x)}{f_q(a+n-1;x)} - \frac{q^{a n} f_q(a;x) f_q(n;x)}{f_q(a+n;x)} \right)
    \\
    =\, &\, \frac{q^n}{1-q^n} \frac{q^{m n} f_q(m;x) f_q(n;x)}{f_q(m+n;x)}
\end{align*}
and setting $m=m_1,\, n=n_1,\, a = m_0$. \end{proof} 
This theorem is the key of proving Theorems \ref{MZV-dual}--\ref{SZ-dual}. Especially, the following corollary will be needed, together with the connection of $Z_q$ with $\bz$ resp. $\sz$ (Remark \ref{cs-rem}).
\begin{corollary}
\label{cs-dagger}
For every SZ-admissible index ${\mathbf{k}}$ and real $x$ with $|x|<1$ we have
\begin{align*}
    Z_q\left({\mathbf{k}};\emptyset ; x\right) = Z_q\left(\emptyset;{\mathbf{k}}^\dagger ;x\right).
\end{align*}
\end{corollary}
\begin{proof}  For all indices ${\mathbf{k}}$ and ${\boldsymbol\ell}$ and $k\geq 1,\, d\geq 0$ we obtain (by $\left(k,\{0\}^{d},{\mathbf{k}}\right)$ we mean the concatination of the indices $\left(k,\{0\}^d\right)$ and ${\mathbf{k}}$) by applying $k$-times \eqref{T2} first and then $(d+1)$-times \eqref{T1}
\begin{align}
\label{cs-szdual-eq}
    Z_q\left(\left(k,\{0\}^{d},{\mathbf{k}}\right);{\boldsymbol\ell} ; x\right) = Z_q\left(\left(\{0\}^{d+1},{\mathbf{k}}\right);\left(\{0\}^{k},{\boldsymbol\ell}\right);x\right)
    = Z_q\left({\mathbf{k}},\left(d+1,\{0\}^{k-1};{\boldsymbol\ell}\right);x\right).
\end{align}
Now, set ${\boldsymbol\ell} = \emptyset$ and write an SZ-admissible index ${\mathbf{k}}$ in the form
\begin{align*}
    {\mathbf{k}} = \left(k_1,\{0\}^{d_1},\dots,k_r,\{0\}^{d_r}\right).
\end{align*}
Then we obtain the corollary by induction on $r$ and using \eqref{cs-szdual-eq} in the induction step.\end{proof} 

With the connection of $Z_q$ and $\sz$ (Rem. \ref{cs-rem} (iii)), SZ-duality follows directly:
\begin{proof}[Proof of Theorem \ref{SZ-dual}]  Take some SZ-admissible index ${\mathbf{k}}$. Using the symmetry of $Z_q$ and setting $x=0$, the claim follows by Corollary \ref{cs-dagger}:
\begin{align*}
    \sz\left({\mathbf{k}}\right) = Z_q\left({\mathbf{k}};\emptyset ;0\right) = Z_q\left(\emptyset;{\mathbf{k}}^\dagger ;0\right) = Z_q\left({\mathbf{k}}^\dagger ;\emptyset ;0\right)
    = \sz\left({\mathbf{k}}^\dagger\right).\tag*{\qedhere}
\end{align*}
\end{proof} 
Analogously, we are able to prove BZ-duality:
\begin{proof}[Proof of Theorem \ref{BZ-dual}]  For an admissible index ${\mathbf{k}}$ we have, using Remark \ref{cs-rem} and Corollary \ref{cs-dagger},
\begin{align*}
    \bz({\mathbf{k}}) &= \lim\limits_{x\rightarrow 1} Z_q\left({\mathbf{k}} - {\mathbf{1}};\emptyset ;x\right) = \lim\limits_{x\rightarrow 1} Z_q\left(\emptyset;({\mathbf{k}} - {\mathbf{1}})^\dagger ;x\right)
    \\
    &= \lim\limits_{x\rightarrow 1} Z_q\left(({\mathbf{k}} - {\mathbf{1}})^\dagger ;\emptyset ;x\right)
    = \bz\left(({\mathbf{k}} - {\mathbf{1}})^\dagger + {\mathbf{1}}\right) 
    = \bz\left({\mathbf{k}}^\vee\right).\tag*{\qedhere}
\end{align*}
\end{proof} 

\begin{example} We give a concrete example of applying transport relations step by step to make clear what happens:
\begin{align*}
        &\, Z_q\left((1,0);\emptyset ;x\right) = Z_q\left((0,0);(0);x\right)
        \\
        =\, &\, Z_q\left((0);(1);x\right) = Z_q\left(\emptyset ; (2);x\right)
\end{align*}
By Remark \ref{cs-rem} (iii) resp. (iv), we obtain $\zeta_q^{SZ}(1,0) = \zeta_q^{SZ}(2)$ resp. $\zeta_q^{BZ}(2,1) = \zeta_q^{BZ}(3)$. We have $(1,0)^\dagger = (2)$ and $(2,1)^\vee = (3)$, why these results indeed correspond to SZ- resp. BZ-duality.
\end{example}

We derive in the following the proof of MZV-duality, Theorem \ref{MZV-dual}, from BZ-duality:
\begin{proof}[Proof of Theorem \ref{MZV-dual}] Let ${\mathbf{k}}$ be any admissible index. Denote by $\wt\left(\mathbf{k}\right) := k_1+\dots+k_r$ the sum of all entries, the \emph{weight} of $\mathbf{k}$. Obviously, one has $\wt\left(\mathbf{k}\right) = \wt\left(\mathbf{k}^\vee\right)$. We have
\begin{align*}
    \zeta\left({\mathbf{k}}\right) = \lim\limits_{q\rightarrow 1} (1-q)^{\wt\left(\mathbf{k}\right)}\bz\left({\mathbf{k}}\right)
    =\lim\limits_{q\rightarrow 1} (1-q)^{\wt\left(\mathbf{k}^\vee\right)} \bz\left({\mathbf{k}}^\vee\right) = \zeta\left({\mathbf{k}}^\vee\right).\tag*{\qedhere}
\end{align*}
\end{proof} 


Consider for the proof of Theorem \ref{q-Ohno} connected sums of the form $Z_q\left({\mathbf{k}};\emptyset ; x\right)$ and the related one of the form $Z_q\left(\emptyset;{\boldsymbol\ell}; x\right)$ using transport relations. In both, we will develop all occurring terms as a Taylor series at $x = 1$, mainly we use that for all $m\in\N$, we have
\begin{align*}
    \frac{1}{1-q^m x} =\,  \frac{1}{1-q^m}\frac{1}{1-\frac{q^m}{1-q^m}(x-1)}
    = \frac{1}{1-q^m}\sum\limits_{c\geq 0} \left(\frac{q^m}{1-q^m}\right)^c (x-1)^c
    =\, \sum\limits_{c\geq 0} \frac{q^{m c}}{(1-q^m)^{c+1}} (x-1)^c.
\end{align*}
\begin{proof}[Proof of Theorem \ref{q-Ohno}]
Let ${\mathbf{k}} = (k_1,\dots,k_r)$ be an admissible index. Then we have
\begin{align*}
    Z_q\left({\mathbf{k}} - {\mathbf{1}};\emptyset ; x\right) =\, & \sum\limits_{m_1>\dots > m_r > 0} \prod\limits_{j=1}^r \frac{1}{1-q^{m_j}x} \frac{q^{m_j (k_j - 1)}}{(1-q^{m_j})^{k_j - 1}}
    \\
    =\, & \sum\limits_{m_1>\dots > m_r > 0} \prod\limits_{j=1}^r \left( \sum\limits_{c_j\geq 0} \frac{q^{m_j c_j}}{(1-q^{m_j})^{c_j+1}} (x-1)^{c_j}\frac{q^{m_j (k_j - 1)}}{(1-q^{m_j})^{k_j - 1}}\right)
    \\
    =\, & \sum\limits_{c_1,\dots,c_r\geq 0} \sum\limits_{m_1>\dots > m_r > 0} \left(\prod\limits_{j=1}^r \frac{q^{m_j (k_j + c_j -1)}}{(1-q^{m_j})^{k_j + c_j}} \right) (x-1)^{c_1 + \dots + c_r}
    \\
    =\, & \sum\limits_{c_1,\dots,c_r\geq 0} \bz({\mathbf{k}} + {\mathbf{c}}) (x-1)^{|{\mathbf{c}}|}.
\end{align*}
Since $\mathbf{k}$ was an arbitrary admissible index and $\mathbf{k}^\vee$ is admissible too, we get
\begin{align*}
    Z_q\left(\emptyset;{\mathbf{k}}^\vee - {\mathbf{1}} ; x\right) = \sum\limits_{c_1,\dots,c_r\geq 0} \bz\left({\mathbf{k}}^\vee + {\mathbf{c}}\right) (x-1)^{|{\mathbf{c}}|}.
\end{align*}

Now, since $Z_q\left({\mathbf{k}} - {\mathbf{1}}; \emptyset ; x\right) = Z_q\left(\emptyset; {\mathbf{k}}^\vee - {\mathbf{1}}; x\right)$ for every admissible index by using the transport relations, the result follows by comparing the coefficient of $(x-1)^c$ on both sides.\end{proof}  

In the same way, we can consider $Z_q\left({\mathbf{k}} - {\mathbf{1}}; \emptyset ; x\right)$ when developing $\frac{1}{1-q^m x}$ around some $a\in\R$, i.e.,
\begin{align*}
    \frac{1}{1-q^m x} = \frac{1}{1-a q^m -q^m (x-a)} = \frac{1}{1-a q^m} \frac{1}{1-\frac{q^m}{1-a q^m} (x-a)} = \sum\limits_{c\geq 0} \frac{q^{m c}}{(1-a q^m)^{c+1}}(x-a)^c.
\end{align*}
Then it is
\begin{align*}
    Z_q\left({\mathbf{k}};\emptyset ; x\right) =\, & \sum\limits_{m_1>\dots > m_r > 0} \prod\limits_{j=1}^r \frac{1}{1-q^{m_j}x} \frac{q^{m_j k_j}}{(1-q^{m_j})^{k_j}}
    \\
    =\, & \sum\limits_{m_1>\dots > m_r > 0} \prod\limits_{j=1}^r  \sum\limits_{c_j\geq 0} \frac{q^{m_j c_j}}{(1-a q^{m_j})^{c_j + 1}} \frac{q^{m_j k_j}}{(1-q^{m_j})^{k_j}} (x-a)^{c_j}.
\end{align*}

\begin{remark}
The series
\begin{align*}
    \sum\limits_{m_1>\dots > m_r > 0} \prod\limits_{j=1}^r  \frac{q^{m_j c_j}}{(1-a q^{m_j})^{c_j + 1}} \frac{q^{m_j k_j}}{(1-q^{m_j})^{k_j}}
\end{align*}
for $c_1,\dots,c_r\geq 0$, $k_1\geq 2,\, k_2,\dots,k_r\geq 1$ and $a\in [0,1]$ can be seen as $q$-analogue of MZVs: For $a=1$ we have seen already by proving the $q$-Ohno relation, how this works. For arbitrary $a$, it is not clear so far, whether we can prove more identities among $q$MZVs with this shape of the connected sum. This could be interesting for the future.
\end{remark}

\section*{Acknowledgements}
I would like to thank to thank Kathrin Bringmann for her helpful comments on the paper. Furthermore, I thank Henrik Bachmann and Ulf K\"uhn for fruitful discussions and lots of comments while supervising my master thesis, of which this paper is part of.

\def\bibindent{1em}

\end{document}